\documentclass[11pt]{amsart}
\usepackage[utf8]{inputenc} 
\usepackage[pctex32]{graphics}
\usepackage{amsfonts}
\usepackage{amssymb,amscd,latexsym}
\usepackage{amsmath}
\usepackage{epsfig}
\usepackage{color}

\newtheorem{thm}{Theorem}[section]
\newtheorem{cor}[thm]{Corollary}
\newtheorem{prop}[thm]{Proposition}
\newtheorem{Proposition}[thm]{Proposition}
\newtheorem{defin}[thm]{Definition}

\newtheorem{lema}[thm]{Lemma}
\newtheorem{rmk}[thm]{Remark}

\newtheorem{ex}{Example}

\newcommand{\Ann}{\operatorname{Ann}}
\newcommand{\Hilb}{\operatorname{Hilb}}

\def\p{\mathbb P}

\def\P{\mathbb{P}}
\def\K{\mathbb{K}}

\def\Hess{\operatorname{Hess}}
\def\hess{\operatorname{hess}}
\def\grad{\bigtriangledown}

\begin{document}

\title{On higher Hessians and the Lefschetz properties}
\author[R. Gondim]{Rodrigo Gondim}
\address{Universidade Federal Rural de Pernambuco, av. Don Manoel de Medeiros s/n, Dois Irmãos - Recife - PE
52171-900, Brasil}
\email{rodrigo.gondim@ufrpe.br}
\date{Maio, 2016}

\begin{abstract}
We deal with a generalization of a Theorem of P. Gordan and M. Noether on hypersurfaces with vanishing (first) Hessian. We prove that for any given $N\geq 3$, $d \geq 3$ and $2\leq k < \frac{d}{2}$ there are 
irreducible hypersurfaces $X = V(f)\subset \P^N$, of degree $\deg(f)=d$, 
not cones and such that their Hessian of order $k$, $\hess^k_f$, vanishes identically. 
The vanishing of higher Hessians is closely related with the Strong (or Weak) Lefschetz property for standard graded Artinian Gorenstein algebra, 
as pointed out first in \cite{Wa1} and later in \cite{MW}. As an application we construct for each pair $(N,d) \neq (3,3),(3,4)$, standard graded Artinian Gorenstein algebras  
$A$, of codimension $N+1 \geq 4$ and with socle degree $d \geq 3$ which do not satisfy the Strong Lefschetz property, failing at an arbitrary step $k$ with 
$2\leq k<\frac{d}{2}$. We also prove that for each pair $(N,d)$, $N \geq 3$ and $d \geq 3$ except $(3,3)$, $(3,4)$, $(3,6)$ and $(4,4)$ there are standard 
graded Artinian Gorenstein algebras of codimension $N+1$, socle degree $d$, with unimodal Hilbert vectors and which do not satisfy the Weak Lefschetz property.
\end{abstract}

\thanks{*Partially supported  by the CAPES postdoctoral fellowship, Proc. BEX 2036/14-2}

\maketitle

\section*{Introduction}

The Weak and the Strong Lefschetz Properties for $\K$-algebras are algebraic abstractions inspired by the Hard Lefschetz Theorem 
on the cohomology rings of smooth complex projective varieties (see \cite{HMMNWW, MN1}). 
Those properties have strong connections with commutative algebra, combinatorics and geometry as one can see in \cite{HMMNWW, MN1, MS}.\\

We will be interested in Gorenstein algebras. A standard graded $\K$-algebra $A$ is Gorenstein if and only if it satisfies the Poincaré Duality Property (see Proposition \ref{G=PD}). This property is the algebraic 
analogue of Poincaré Duality Theorem for the cohomology ring (see Definition \ref{defArtgrad} and Proposition \ref{G=PD}). Since we are interested in construct algebras failing SLP,  by Lefschetz's hard Theorem (see \cite{La}) we are askink for algebras 
that can not occurs as cohomology rings of smooth projective varieties. \\

The more systematic way to produce examples of graded Artinian Gorenstein algebra not satisfying the WLP seems to be the construction of algebras having nonunimodal Hilbert vector. 
 Stanley in \cite{St2} constructed the first example of a nonunimodal Artinian Gorenstein algebra. After that, others authors studied nonunimodality and its implications (\cite{BI,Bo,BoL,MNZ}). 
 In \cite{MRO} the authors introduced a geometric approach which allowed them to produce other examples of Artinian Gorenstein algebra not satisfying the Weak Lefschetz property. 
The examples obtained in \cite{Ik} and \cite{MW} consist of standard graded Artinian Gorenstein algebras having unimodal Hilbert vector and not satisfying the WLP. These examples inspired us
although they seemed to be sporadic and, apparently, without a strong motivation explaining their constructions.\\

Our goal in this paper is  to present new families of algebras which do not satisfy the SLP or the WLP. Our main results are Theorem \ref{cor:question1}, Theorem \ref{thm:watanabecomponent}, Corollary \ref{cor:wlpodd} and Theorem \ref{thm:wlp}.
The first and the second one are generalizations of Gordan-Noether's Theorem on the existence of non trivial forms with vanishing hessians the other two are aplications on algebras failing SLP or WLP. 
The families constructed, in general, have unimodal Hilbert vectors (see  Theorem \ref{thm:wlp}).
The constructions of  these families were motivated by the Hessian criterion for Lefschetz elements (see \cite{Wa1,MW}), according to which the vanishing of any $k$th Hessian 
of a form implies that the associated algebra does not satisfy the SLP. Therefore, we construct families of forms whose $k$th Hessian vanish identically 
(see Theorem \ref{cor:question1}, Proposition \ref{prop:formacanonica} and Theorem \ref{thm:watanabecomponent}).\\

Hypersurfaces with vanishing Hessian were first studied by O. Hesse who claimed twice they must be cones (\cite{Hesse1,Hesse2}).  
P. Gordan and M. Noether proved in \cite{GN}  that for arbitrary degree $d \geq 3$ Hesse's claim is true for $N \leq 3$ (c.f. Theorem \ref{thm:GN}) and false for $N\geq 4$, by constructing a  series of 
counterexamples in $\P^N$ for each $N\geq 4$ and for each $d\geq 3$ (see \ref{thm:GN2}).
In Corollary \ref{cor:GNgeneral} we prove a generalization of this result for higher Hessians. 
Moreover Gordan and Noether also classify all hypersurfaces with vanishing Hessian in $\p^4$ showing that they are either cones or belong to their series. 
This classical work has been revisited in (\cite{CRS,dB,Lo,GR,Wa2,Wa3,Ru}). 
In \cite{Pe} U. Perazzo studied the case of cubic hypersurfaces with vanishing Hessian and his contributions have been rewritten in modern terms in \cite{GRu}. \\

Proposition \ref{prop:key} replaces the strategy for the construction of all known hypersurfaces with vanishing Hessian. In this proposition 
we give a sufficient condition for the vanishing of $k$th Hessian which is shared by all known examples.   
Although many counterexamples to  WLP are known 
for each codimension $N+1 \geq 3$, at the best of our knowledge no series of examples for each socle degree $d=\deg(f)$ has been constructed so far. 
One of our main results is Theorem \ref{thm:wlp} where we show that for each pair $(N,d)$ with  $N\geq 3$, and with $d \geq 3$, except  $\{(3,3), (3,4), (4,4), (3,6)\}$, there exist standard graded Artinian Gorenstein 
algebras of codimension $N+1$ and socle degree $d$ that do not satisfy the WLP. In the cases $(3,3)$ and $(3,4)$ the algebra satisfy the SLP by Gordan-Noether Theorem \ref{thm:GN}, 
and the case $(4,4)$ is treated separately in Proposition \ref{prop:44}.\\
 
The first family considered is a broad generalization of an example in codimension $4$ due to Ikeda (see \cite{Ik,MW}). The second class of examples is a generalization of the idea used to construct the GN polynomials in \cite{GN} 
which are of the same type as those treated in \cite{Pe} (see Proposition \ref{prop:formacanonica} and \cite{CRS}). \\


 Now we describe in more detail the contents. The first section is devoted to preliminary results, including the basic definitions of gradient and Hessian of 
order $k$. We introduce the Lefschetz properties for graded Artinian algebras and we focus on  standard graded 
Gorenstein Artinian algebras, which are our main object of analysis. To this aim we also recall a useful characterization of standard graded Artinian Gorenstein $\K$-algebras, 
Theorem \ref{G=ANNF}. Finally we state the Hessian criterion due to Watanabe, Theorem \ref{WHessian}, connecting the previous subjects. \\

The second section contains the main constructions, two families of forms with vanishing higher hessian generalizing Gordan-Noether's theorem (c.f. Theorem \ref{thm:GN2}). 
We call the polynomials in the first family the exceptional and those in the second family GNP-polynomials (where GNP stands for Gordan-Noether-Perazzo).
The two families are of different nature and the second one includes the case $k=2$. Corollary \ref{cor:GNgeneral} summarizes all the constructions of this section. \\

In the third section we apply the previous results to construct algebras not satisfying the SLP and/ or the WLP. 
Thus, Theorem \ref{thm:GN2} can also be translated in a result concerning algebras not satisfying SLP (see Corollary \ref{cor:slp}). 
Theorem \ref{thm:wlp} deals with algebras not satisfying the WLP but having a unimodal Hilbert vector, showing that there are such examples for arbitrary
socle degree $d\geq 3$.\\

For reader's convenience we recall the main results of Gordan-Noether's theory in an Appendix. The most detailed account on this theory is \cite[Chapter 7]{Ru}. 
We recall  two fundamental results due to Gordan and Noether, Theorem \ref{thm:GN} and Theorem \ref{thm:GN2} 
dealing with Hesse's claim for Hessian in the usual sense. We also survey the classical constructions due to 
Gordan-Noether, Permutti and Perazzo of families of hypersurfaces not cones and with vanishing Hessian.

\section{Preliminaries: Higher Hessians and the Lefschetz properties}


In this paper $\K$ denotes a field of characteristic zero. Let  $R=\K[x_0,\ldots,x_N]$ be the polynomial ring in $N+1$ variables and let  
$R_d = \K[x_0,\ldots,x_N]_d$ be the $\K$ vector space of homogeneous forms of degree $d$. The standard $\K$-basis of $\K[x_0,\ldots,x_N]_d$ is 
$$\mathcal{B}=\{\displaystyle \prod_{i=0}^Nx_i^{e_i}|\ e_0+\ldots +e_N=d\},$$
yielding the well known formula  $\dim_{\K}\K[x_0,\ldots,x_N]_d = \binom{N+d}{d}$. 
\medskip

\begin{defin}
If $R = \K[x_0,\ldots,x_N]$, we denote by $$Q=\K[X_0, ,\ldots,X_N]$$ the ring of differential operators
on $R$, where $X_i : = \frac{\partial}{\partial x_i}$.

For each $d \geq k\geq 0$ there exist natural $\K$-bilinear maps $R_d \times Q_k \rightarrow R_{d-k}$ defined by differentiation, $(f,\alpha) \mapsto f_{\alpha}:=\alpha(f)$. 
\end{defin}

\begin{rmk}\rm \label{rmk:dualbasis}
Notice that for $\alpha \in Q$ and $f \in R$ we use the two notations $f_{\alpha}$ and $\alpha(f)$ interchangeably meaning the differential operator $\alpha$ acting in $f$.\\

For each $d \geq 1$ the $\K$-bilinear map $R_d \times Q_d \to \K$ is non degenerate, hence there is a natural identification $Q_d \simeq R_d^*$. 
The duality implies that for each set of linearly independent forms $f_1,\ldots,f_s \in R_d$, there are differential operators $\alpha_1,\ldots,\alpha_s \in Q_d$ 
such that $\alpha_i(f_j) = \delta_{ij}$, the Kronecker's delta. 
\end{rmk}

\begin{defin}

Let $f\in R = \K[x_0,\ldots,x_N]$ be a polynomial and let $k\geq 1$. 
If $\mathcal{B}=\{\alpha_1,\ldots,\alpha_{\nu}\} \subset \K[X_0,\ldots,X_N]_k$ is any ordered basis of $Q_k$, $\nu= \nu(N,k)=\binom{N+k}{k}$, we define the $k$th gradient of $f$ with 
respect to the basis $\mathcal{B}$ by $$\bigtriangledown_{\mathcal{B}}^k f=(\alpha_1(f),\ldots,\alpha_{\nu}(f)).$$ If the basis is clear from the context or if it is the standard basis ordered in the 
lexicographical order we put $\bigtriangledown^k f$ to denote the $k$th gradient with respect to thas basis.
\end{defin}

\begin{ex}\rm \label{ex:gradikeda}
Let $g_0=u^3,g_1=u^2v, g_2=uv^2, g_3=v^3 \in \K[u,v]$. We have $\grad^2{g_0}= (6u,0,0)\sim(1,0,0)$, $\grad^2{g_1}= (2v,2u,0)\sim(1,A,0)$, 
$\grad^2{g_2}= (0,2v,2u)\sim(0,1,B)$ and $\grad^2{g_3}= (0,0,6v)\sim(0,0,1)$ with $A,B \in \K(u,v)$. Here $\sim$ denotes parallel vectors over the field of fractions. They are linearly dependent over $\K(u,v)$.  
\end{ex}

\begin{rmk}\rm \label{rmk:key}
We are interested in identifying  sets of linearly independent forms $g_1,\ldots,g_s \in \K[u_1,\ldots,u_m]_{d}$ whose $k$th gradients are linearly dependent over the field of fractions $\K(u_1,\ldots,u_m)$. As we 
will see in the sequel, this construction is related to the vanishing of the $k$th Hessian. 
Notice that given $g_1,\ldots,g_s \in \K[u_1,\ldots,u_m]_d$,  if $$s >\binom{k+m-1}{k}=\dim_{\K} \K[U_1,\ldots,U_m]_k,$$ 
 then the $k$th gradients $\grad^k g_1,\ldots,\grad^k g_s$ are linearly dependent over $\K(u_1,\ldots,u_m)$, the field of fractions.  
\end{rmk}

Let $f \in R = \K[x_0,\ldots,x_N]_d$ be an irreducible homogeneous polynomial of degree $\deg(f)=d \geq 1$ and let 
$Q=\K[X_0,\ldots,X_N]$ be the ring of differential operators. 
We define $$\operatorname{Ann}(f) = \{\alpha \in Q| \alpha(f)=0\}\subset Q.$$
Since $\operatorname{Ann}(f)$ is a homogeneous ideal of $Q$, we can define $$A=\frac{Q}{\operatorname{Ann}(f)}.$$
Therefore $A$ is a standard graded Artinian Gorenstein $\K$-algebra such that $A_j=0$ for $j>d$ and such that $A_d \neq 0$, (\cite[Section 1,2]{MW}). 
We assume, without loss of generality, that $(\operatorname{Ann}(f))_1=0$.




Conversely, by the Theory of Inverse Systems developed by Macaulay, we get the following characterization of standard
Artinian Gorenstein graded $\mathbb K$-algebras. A proof of this result can be found in \cite[Theorem 2.1]{MW} (see also\cite{MS}).

\begin{thm}{\bf \ ( Double annihilator Theorem of Macaulay)} \label{G=ANNF} \\
Let $R = \K[x_0,x_1,\ldots,x_N]$ and let $Q = \K[X_0,\ldots, X_N]$ be the ring of differential operators. 
Let $A= \displaystyle \bigoplus_{i=0}^dA_i = Q/I$ be an Artinian standard graded $\mathbb K$-algebra. Then
$A$ is Gorenstein if and only if there exists $f\in R_d$
such that $A\simeq Q/\operatorname{Ann}(f)$.
\end{thm}

\begin{defin}
 Let $A= \displaystyle \bigoplus_{i=0}^dA_i = Q/I$ be an Artinian Gorenstein $\K$-algebra. The socle degree of $A$ is $d$ which coincides with the degree of the form $f$. 
 The codimension of $A$ is the codimension of the ideal $I \subset Q$ which coincides with its embbed dimension.
\end{defin}

\begin{defin}
 Let $f \in \K[x_0,\ldots,x_N]$ be a homogeneous polynomial, let $A=\frac{Q}{\operatorname{Ann}(f)}$ be the associated Artinian Gorenstein algebra and 
 let $\mathcal{B}=\{\alpha_j|j=1,\ldots,\sigma_k\} \subset A_k$ be an ordered $\K$-basis. The $k$th (relative) Hessian matrix of $f$ with respect to $\mathcal{B}$ is 
 $$\Hess_f^k = (\alpha_i\alpha_j(f))_{i,j}^{\sigma_k}.$$
 The $k$th Hessian of $f$ with respect to $\mathcal{B}$ is
$$\hess_f^k = \det (\Hess_f^k).$$
 \end{defin}

\begin{rmk}\rm The Hessian of order $k=0$ is just $\hess^0_f=f$. 
 The Hessian of order $k=1$ with respect to the standard basis is just the classical Hessian. 
 Although the definition of  Hessians of order $k$ depends on the  choice of a  basis of $A_k$, the vanishing of the $k$th Hessian is independent from this choice.
 More precisely a basis change has the effect of multiplying  the determinant by a non-zero element of the base field $\K$. 
Since we are interested in the vanishing of the Hessian we do not attach the basis in the notation of absolute higher Hessian.   
\end{rmk}

 
We now define the Lefschetz Properties (see \cite{HMMNWW}). 

\begin{defin}\label{WLPSLP}
Let $\mathbb K$ be a field and let $$A=\bigoplus_{i=0}^dA_i$$ be an Artinian graded $\mathbb K$-algebra with $A_d\neq 0$.

The algebra  $A$ is said to have {\it the Strong Lefschetz Property}, briefly $SLP$, if there exists an element $L\in A_1$ such that the multiplication map
$$\bullet L^k: A_i\to A_{i+k}$$

is of maximal rank for $0\leq i\leq d$ and $0\leq k\leq d-i$. 

\medskip

The algebra  $A$ is said to have {\it the Weak Lefschetz Property}, briefly $WLP$, if there exists an element $L\in A_1$ such that the multiplication map
$$\bullet L: A_i\to A_{i+1}$$

is of maximal rank for $0\leq i\leq d-1$.


 
 $A$ is said to have {\it the Strong Lefschetz Property in the narrow sense} if there exists an element $L\in A_1$ such that the multiplication map
$$\bullet L^{d-2i} : A_i \to A_{d-i}$$
 is an isomorphism for $i = 0,\ldots,[\frac{d}{2}].$
 \end{defin}

These Lefschetz properties were inspired by the Hard Lefschetz Theorem for the cohomology ring of complex projective manifolds. 
We recall that such  cohomology groups also satisfy the so called Poincaré Duality.

\begin{defin}\label{defArtgrad} Let $\K$ be a  field and let $$A=\bigoplus_{i=0}^dA_i$$ be an Artinian graded $\K$-algebra with $A_0=\mathbb K$ and
$A_d\neq 0$. Let 
$$
\begin{array}{cccc}
\bullet: &A_i\times A_{d-i}&\to& A_d\\
&(\alpha,\beta)&\to&\alpha\bullet \beta
\end{array}
$$
be the restriction of the multiplication in $A$.

We say that $A$ {\it satisfies the Poincaré Duality Property} if:
\begin{enumerate}
\item[(i)] $\dim_{\mathbb K}(A_d)=1;$
\medskip
\item[(ii)] $\bullet :A_i\times A_{d-i}\to A_d\simeq\mathbb K$ is non-degenerate for every $i=0,\ldots,[\frac{d}{2}].$
\end{enumerate}

\end{defin}
\medskip

The algebra $A$ is said to be {\it standard} if 
$A\simeq \frac{\mathbb K[x_0,\ldots, x_N]}{I}$,
as graded algebras, with $I\subset \mathbb K[x_0,\ldots, x_N]$ a homogeneous ideal and the degrees of $x_o, x_1, \ldots, x_N$ are one. 

\begin{Proposition}\label{G=PD}{\rm (\cite{GHMS}, \cite[Proposition 2.1]{MW}), \cite[Proposition 1.4]{MS}} Let $A$ be a graded Artinian $\mathbb K$-algebra.
Then $A$ satisfies the Poincaré Duality Property if and only if it is Gorenstein.
\end{Proposition}

\begin{defin}
To each Artinian graded $\mathbb K$ algebra $A=\displaystyle \bigoplus_{i=0}^dA_i$ as above, letting $h_i=\dim_{\mathbb K}A_i$,  we can associate its Hilbert vector
$$\operatorname{Hilb}(A)=(1,h_1,\ldots,h_d).$$
Algebras satisfying the Poincaré Duality property have symmetric Hilbert vector, that is, $h_{d-i}=h_i$ for every $i=0,1,\ldots,[\frac{d}{2}]$.
The Hilbert vector of $A$ is said to be unimodal if there exists an integer $t\geq 1$ such that
$$1\leq h_1\leq\ldots\leq h_t\geq h_{t+1}\geq\ldots \geq h_{d-1}\geq 1.$$
\end{defin}
\medskip

The fundamental link between the study of the Lefschetz properties and the higher Hessians is the following Theorem due to Watanabe (see \cite{Wa1,MW}).

\begin{thm}\label{WHessian} {\rm (\cite{Wa1}, \cite{MW})} Let notation be as above.
An element $L = a_0x_0+\ldots+a_Nx_N\in A_1$ is a strong Lefschetz element of $A = Q/\operatorname{Ann}(f)$ if and only if $\hess^k_f(a_0,\ldots, a_N)\neq 0$ for all $k=0,\ldots, [d/2]$.
\end{thm}


\begin{rmk}\rm
It is not difficult to see that for a standard graded Artinian Gorenstein algebra $A = Q/\operatorname{Ann}(f)$ the notion of $SLP$ and $SLP$ {\it in the narrow sense} coincide.   \\
For $N=1$ (codimension $2$) all Artinian graded algebras satisfy the SLP (\cite{HMNW}). Thus all polynomials of degree $d$ in two variables have 
non-vanishing $k$th Hessian for all $k < \frac{d}{2}$. \\ 
For $N=2$ (codimension $3$) it is an open problem to know if the SLP (or the WLP) holds or if there exist Artinian Gorenstein algebra not satisfying the SLP (or the WLP). 
In \cite{BMMNZ} the authors reduced the problem of the WLP to the so called compressed algebras, more precisely they prove that if all compressed standard graded Artinian Gorenstein  
of codimension three satisfy the WLP, then all standard graded Artinian Gorenstein of codimension three satisfy the WLP. 
\end{rmk}

\section{Families of forms having vanishing $k$th Hessian}

The aim of this section is to prove Theorem \ref{cor:question1} which is a generalization of Gordan-Noether Theorem \ref{thm:GN2}. In order to do this we deal with the constructions of two 
families of polynomials having $k$th vanishing Hessian. To construct these families we look for a higher order Gordan-Noether criterion, Proposition \ref{prop:GNcriteria}, 
at least giving a sufficient condition to the vanishing of the $k$th Hessian. \\

The unifying point of view can be summarized in the next Proposition, which is the core of our subsequent constructions. At the best of our  knowledge all known examples of polynomials whose $k$th Hessian vanishes identically, for some $k\geq 1$, either satisfy this property 
up to a linear change of coordinates or $k=1$ and the polynomial is built up with separated variables using Perazzo polynomials , that satisfy this property (c.f. Appendix \ref{ap}, \cite{CRS} or \cite[Chapter 7]{Ru}). 

\begin{prop}\label{prop:key} Let  $R = \K[x_0,\ldots,x_n,u_{1},\ldots,u_m]$ be a polynomial ring in $m+n+1$ variables and   
$Q=\K[X_0,\ldots,X_n,U_{1},\ldots,U_m]$ be the associated ring of differential operators. Let $f \in R_d$ be a form of degree $d$, $k < d/2$ an integer and $A =A(f)= Q/\operatorname{Ann}(f)$.
Suppose that there exists $s$ monomials $\alpha_1, \alpha_2, \ldots, \alpha_s \in Q_k \setminus \K[U_1, \ldots, U_m]_k$ linearly independent in $A_k$ such that $\alpha_i(f) \in \K[u_1, \ldots, u_m]$. 
If $s > \binom{m+k-1}{k}$, then $$\hess^k_f=0.$$

\end{prop}
\begin{proof} Let $\tilde{Q} = \K[U_1,\ldots,U_m]$ and $B= \tilde{Q}/\operatorname{Ann}(f) \cap \tilde{Q}$.
Let us construct a monomial ordered basis of $A_k$, $$\mathcal{A} = \{\alpha_1, \ldots, \alpha_s, \gamma_1, \ldots, \gamma_l, \beta_1, \ldots, \beta_r\}$$ The first $s$ vectors are $\alpha_1,\ldots,\alpha_s$, 
 and the last vectors $\{\beta_1,\ldots,\beta_r\}$ consist of a basis $\mathcal{B}$ of $B_k$.  
 Let $$\grad^k \alpha_i(f)=(\beta_1(\alpha_i(f)), \ldots,\beta_r(\alpha_i(f))) = (\alpha_i(\beta_1(f)), \ldots,\alpha_i(\beta_r(f)))$$ be the gradient of $\alpha_i(f)$ with respect to the basis $\mathcal{B}$. 
 For $i=1,\ldots,s$, the first $s$ rows of $Hess^k_f$ are  
 $$L_i=(0,\ldots,0,\grad^k \alpha_i(f)).$$
 Indeed, for any differential operator $\delta \in \{\alpha_1, \ldots, \alpha_s, \gamma_1, \ldots, \gamma_l\}$ it depends on some of the variables $X_0, \ldots,X_n$. Since $\alpha_i(f) \in \K[u_1,\ldots,u_m]$, we get 
 $\alpha_i(\delta(f))=\delta(\alpha_i(f))=0$.

 By hypothesis, $s>\binom{m-1+k}{k}=\dim \K[U_1,\ldots,U_m]_k$, hence the $k$th gradients of the $\alpha_i(f)$, $\grad^k \alpha_1(f),\ldots,\grad^k \alpha_s(f)$, are linearly dependent 
 over $\K(u_1,\ldots,u_m)$. Therefore  $L_1,\ldots,L_s$ are linearly dependent over $\K(x_0,\ldots,x_n,u_1,\ldots,u_m)$ yielding $\hess^k_f=0$. 
 In fact, we can think on $\Hess^k_f$ as a matrix with entries in $\K(x_0,\ldots,x_n,u_1,\ldots,u_m)$, thus a necessary and sufficient condition for the vanishing 
 of the Hessian is actually the linear dependence among its rows over the field of fractions $\K(x_0,\ldots,x_n,u_1,\ldots,u_m)$.
\end{proof}








Let us revisit an example due to Ikeda \cite{Ik} and \cite{MW} from the previous perspective.

\begin{ex} \rm \label{ex:Ikeda}
 Let $f=x_0u_1^3u_2+x_1u_1u_2^3+x_0^3x_1^2 \in \K[x_0,x_1,u_1,u_2]_5$. According to the previous notation we have $m=2$, $k=2,e=3,d=5$.  
 For this polynomial we have $\hess_f \neq 0$ and $\hess^2_f=0$. \\
 Indeed, it is easy to see that the first partial derivatives of $f$ are linearly independent, so by  Theorem \ref{thm:GN}, $\hess_f\neq 0$. 
 On the other hand, $f_{X_0U_2}=u_1^3,f_{X_0U_1}=3u_1^2u_2,f_{X_1U_2}=3u_1u_2^2, f_{X_1U_1}=u_2^3 \in \K[u_1,u_2]$. Therefore, $\hess^2_f=0$,  by Proposition \ref{prop:key}. 
Indeed, $\dim A_2 =10ky$ and following the strategy of Proposotion \ref{prop:key}, we can choose an ordered basis for $A_2$ starting with $\{X_0U_2,X_0U_1,X_1U_2,X_1U_1\}$ and 
completed by $\mathcal{B} = \{U_1^2, U_1U_2,U_2^2\}$. For instance we can choose the basis $\{X_0U_2,X_0U_1,X_1U_2,X_1U_1,X_0^2,X_0X_1,X_1^2,U_1^2, U_1U_2,U_2^2\}$.
Let us denote $f_1 = f_{X_0U_2}=u_1^3,f_2=f_{X_0U_1}=3u_1^2u_2,f_3=f_{X_1U_2}=3u_1u_2^2, f_4=f_{X_1U_1}=u_2^3$. 

%
  %
 
 Since $s=4>3\binom{2+2-1}{1}$, by Proposition \ref{prop:key} $\hess^2_f  = 0$.\\

\end{ex}

The first family we construct is a generalization of Ikeda's Example.\\

Let $M_1, \ldots, M_n$ be monomials in $\K[u,v]_{d-1}$. Let $\mathbb{V} = \K[U,V]_{k-1}\{M_1, \ldots, M_n\}$ be the vector space spanned by the $(k-1)$th derivatives of the monomials 
$M_i$ with respect to $u,v$. 
%
 
\begin{defin} With the previous notation, if $\dim \mathbb{V} > k+1$, then a plolynomial of the form $f = \sum x_i M_i + h(x) \in \K[u,v,x_1,\ldots,x_n]$ is called an exceptional polynomial of degree $d$ 
and order $k$. 
\end{defin}

\begin{thm} \label{cor:question1} For each $n \geq 3$, for each $d \geq 5$ and for $2\leq k < \frac{d}{2}$ there exist irreducible exceptional polynomials $f \in \K[u,v,x_2,\ldots,x_n]$ 
of degree $\deg(f) = d$ such that $$\hess_f\neq 0\ \text{and}\  \hess_f^r=0 \ \text{for}\ r=2,\ldots,k.$$
Furthermore, if $k+1 \leq \frac{d}{2}$, then $\hess_f^{k+1}\neq 0$.
\end{thm}

\begin{proof}
Consider $f=g(u,v,x_2,x_3)+h(x_2,x_3)+p(x_4,\ldots,x_n)$ with $g=x_2u^{k-1}v^{d-k}+x_3u^{d-2}v$, and let $h$ and $p$ be chosen to make $f$ irreducible.
Let $\tilde{f}=g+h$, and consider $\tilde{X} \subset \P^3$. For a general $h$ one can check that $\tilde{f}$ does not define a cone in $\P^3$, 
since its first partial derivatives are linearly independent. By Gordan Noether Theorem,  Theorem \ref{thm:GN} here, we have $\hess_{\tilde{f}} \neq 0$.
Notice that $$\Hess_f = \left[ \begin{array}{cc}
                                        \Hess_{\tilde{f}} & 0\\
                                        0 & \Hess_p
                                       \end{array} \right]. $$
Since $\hess_{\tilde{f}}\neq 0$ and $\hess_p \neq 0$ for general $p$, one concludes that $\hess_f \neq 0$ for a general $f$ of this type.\\                                       
On the other hand, for each $r \leq k$ we consider $\alpha_j=X_2U^{r-1-j}V^j$ with $j=0,\ldots,r-1$. Thus $f_{\alpha_j} = a_ju^{k-r+j}v^{d-k-j}\in \K[u,v]_{d-r}$. 
Consider $\beta=X_3U^{r-2}V$ and $\gamma=X_3U^{r-1}$ so that $f_{\beta}=bu^{d-r},f_{\gamma}=cu^{d-r-1}v \in \K[u,v]_{d-r}$. 
Recall that the linear independence of $\{\alpha_0, \ldots, \alpha_{r-1}, \beta, \gamma\}\subset A_r$ is equivalent to the linear independence of the derivatives 
$\{f_{\alpha_0}, \ldots, f_{\alpha_{r-1}}, f_\beta, f_\gamma\}\subset \K[u,v]_{d-r}$.
To conclude their linear independence it is enough to verify that neither $f_{\beta} $ nor $f_{\gamma}$ is a scalar multiple of $ f_{\alpha_j}$ for $j=0,\ldots,r-1$. 
If this were the case, one would deduce either $j=d-k \leq r-1$, yielding $ d < d-1$, or $j=d-k-1$ implying $d < d$ and one would get a contradiction in both cases.

Since $\dim \K[U,V]_r=r+1$ and since we  found $r+2$ linearly independent differentials $\{\alpha_0,\ldots,\alpha_{r-1},\beta,\gamma\} \subset A_r$, we get $\hess^r_f=0$ by 
Proposition \ref{prop:key}.

To conclude the proof we must show that $\hess^{k+1}_f \neq 0$ for the general $f$ if $j+1\leq\frac{d}{2}$. 
Consider $f=g+h+p$, then 
$$\Hess^{k+1}_f = \left[ \begin{array}{ccc}
                                        \Hess_g^{k+1} & 0 & 0\\
                                        0 & \Hess_h^{k+1} & 0\\
                                        0& 0 & \Hess_p^{k+1}
                                       \end{array} \right]. $$
Since $\hess^{k+1}_h\neq 0$ and $\hess_p^{k+1}\neq 0$ for general $h,p$, it is enough to prove that $\hess_g^{k+1} \neq 0$. 
Let $Q = \K[U,V,X_2,X_3]$ be the ring of differential operators and consider $A = Q/(\operatorname{Ann}(g))$ for $g=u^{k-1}x_2v^{d-k}+vx_3u^{d-2}$.
Notice that $\dim A_{k+1}=2k+4$ since a ordered $\K$-basis for $A_{k+1}$ is\\
$$\mathcal{B} = \{\alpha_1,\alpha_2,\alpha_3,\alpha_4,\beta_0 ,\gamma_0,\beta_1,\gamma_1,\ldots \beta_i,\gamma_i,\ldots, \beta_{k-1},\gamma_{k-1}\}.$$

Where $$\alpha_1=U^{k+1}$$ $$\alpha_2=U^kX_3$$ $$\alpha_3=U^kV$$ $$\alpha_4=U^{k-1}VX_3$$ And for $i=0, \ldots,  k-1$.

$$\beta_i=V^{k+1-i}U^i$$ $$\gamma_i=V^{k-i}X_2U^i$$

The matrix $\Hess^{k+1}_g$ can be partitioned in blocks, induced by the partition of the basis $\mathcal{B}$ by choosing the first four vectors $\{\alpha_1,\alpha_2,\alpha_3,\alpha_4\}$ 
and the $2k$ other ones.  
$$\Hess^{k+1}_g = \left[ \begin{array}{cc}
                                        \varTheta_{4 \times 4} & 0_{4 \times 2k }\\
                                        0_{2k \times 4} & \Delta_{2k \times 2k}
                                       \end{array} \right]. $$
The zero in the block anti-diagonal follows from $\alpha_i\beta_j=U^{k+j-i+2}V^{k+1-j}X_3^{i-1} \in \operatorname{Ann}(g)$ and \\ 
$\alpha_i \gamma_j = U^{k+j-i+2}V^{k-j}X_2X_3^{i-1} \in \operatorname{Ann}(g)$ 
for every $i=1,2,3,4$ and $j=0,\ldots,k-1$. \\
We claim that $$\varTheta_{4 \times 4} = \left[ \begin{array}{cccc}
                                        \ast & \ast & \ast & \ast \\
                                        \ast & \ast & \ast & 0 \\
                                       \ast & \ast & 0 & 0 \\
                                       \ast & 0 & 0 & 0
                                       \end{array} \right]. $$
With the elements of the off diagonal non-zero, hence, $\det \varTheta \neq 0$. Indeed the elements of the off diagonal are, $\alpha_1\alpha_4=\alpha_2\alpha_3=U^{2k}VX_3 \not \in \operatorname{Ann}(g)$ and 
the elements of the lower triangle $\alpha_2\alpha_4,\alpha_3^2,\alpha_3\alpha_4,\alpha_4^2 \in \operatorname{Ann}(g)$. \\
In the same way $$\Delta_{2k \times 2k} = \left[ \begin{array}{cccc}
                                        \ast & \ast & \ldots & \ast \\
                                        \ast & \ldots & \ast & 0 \\
                                       \ldots & \ldots & 0 & 0 \\
                                       \ast & 0 & \ldots & 0
                                       \end{array} \right]. $$
In fact, the off lower triangle is zero since $\beta_i\gamma_j = U^{i+j}V^{2k+1-i-j}X_2 \in \operatorname{Ann}(g)$ if $i+j>k-1$. 
On the contrary, the elements of the off diagonal are non-zero, because
they are $\beta_i \gamma_{k-1-i}=V^{k+2}U^{k-1}X_2 \not \in \operatorname{Ann}(g)$. Therefore $\det \Delta \neq 0$ and the result follows.                                    
\end{proof}


\begin{rmk}\rm We want to stress that, 
by Gordan-Noether's Theorem, this class of forms with vanishing higher Hessians starts with four variables which does not occurs for the classic Hessian (c.f. Theorem \ref{thm:GN}). So, exceptional forms are actually of different nature and not associated to Gordan-Noether's construction.
\end{rmk}
 
The second family we construct was inspired by the Gordan-Noether's and Perazzo's polynomials. They are called GNP-polynomials of type $(m,n,k,e)$ ( see Proposition \ref{prop:formacanonica}). 
They are a natural generalization of Perazzo's polynomials; for instance, any GNP-polynomial of type $(m,n,1,e)$ is a Perazzo polynomial (c.f. \cite{Pe, GRu}). 
They are also a generalization of some special cases of GN polynomials, more precisely, the case $\mu=1$ in Definition \ref{def:GN} and 
the general case, assuming $P_j=0$ for $j \neq 0, \mu$ (c.f. \cite{CRS} or \cite[Chapter 7]{Ru}). GNP-polynomials also generalize some examples due to Maeno and Watanabe, (c.f. \cite[p.10, Example 5.1]{MW} and \cite[p. 11, Example 5.2]{MW}).  

\begin{prop} \label{prop:formacanonica} Let $x_0,\ldots,x_n$ and $u_1,\ldots,u_m$ be independent sets of indeterminates with $n+1 \geq m \geq 2$. 
For $j=1,\ldots,s$, let $f_j \in \K[x_0,\ldots,x_n]_k$ and $g_j\in \K[u_1,\ldots,u_m]_e$ be linearly independent forms with $1\leq k < e $. 
If $s>\binom{m-1+k}{k}$, then the form of degree $d=e+k$ given by 
\begin{equation}\label{GNP}
f=f_1g_1+\ldots+f_sg_s 
\end{equation}
satisfies  
$$\hess^k_f=0.$$
Let $A=Q/\Ann(f)$ with $f$ of type \ref{GNP} and suppose that $\dim A_1 = m+n+1$. If $\hess^k_f=0$ then $f$ is called a $GNP$-polynomial of type $(m,n,k,e)$. 

\end{prop}

\begin{proof} Let $R = \K[x_0,\ldots,x_n,u_1,\ldots,u_m]$,
let $Q=\K[X_0,\ldots,X_n,U_1,\ldots,U_m]$ be the associated ring 
of differential operators and let $A =A(f)= Q/\operatorname{Ann}(f)$ be the associated Artinian Gorenstein algebra.

 Consider a basis of $A_k$ whose first $s$ vectors $\alpha_1,\ldots,\alpha_s$ are the dual of $f_1,\ldots,f_s$ in the sense of Remark \ref{rmk:dualbasis}, that is 
 $\alpha_i(f_j)=\delta_{ij}$. Notice that $\alpha_j(f)= g_j \in \K[u_1,\ldots,u_m]$ for $j=1,\ldots,s$, and that by hypothesis $s>\binom{m-1+k}{k} = \dim \K[U_1,\ldots,U_m]_k$.
  Thus $\hess^k_f=0$ by Proposition \ref{prop:key}.
\end{proof}
\medskip
 
Now we prove the existence of families of GNP-polynomials of type $(m,n,k,e)$ for every codimension $N+1 = m+n+1 \geq 5$ and for every degree $d= e+k \geq 3$. 
Our strategy is to determine the possible values of $\dim A_1$ for GNP-hypersurfaces of type $(m,n,k,e)$ with $m\geq 2$. Since we deal with a separation of the set of variables in two subsets with 
different roles, we call the $u_1,\ldots,u_m$ essential variables and $x_0,\ldots,x_n$ superfluous variables. 
\medskip

\begin{defin}\label{defin:Am}
 Set $$\mathcal{A}_{k,e}^m=\{\dim A_1| A = Q/\Ann(f), f \ \text{is a GNP-polynomial of type}\ (m,n,k,e)\}.$$
 Denote $a_m = a_m(k,e) = \min \mathcal{A}_{k,e}^m$ and $b_m = b_m(k,e) = \max \mathcal{A}_{k,e}^m$.
\end{defin}
\medskip

\begin{lema} \label{lema:m=2}
 $$\mathcal{A}^2_{k,e}=\{5,6,\ldots,e+3\}$$
\end{lema}

\begin{proof} By Proposition \ref{prop:formacanonica}, since $\dim \K[u,v]_k=k+1$, it is enough to exhibit $k+2$ linearly independent $g_j \in \K[u,v]_e$.
Let $ g_j=u^{e-j}v^j$ for $j=0,1,\ldots,k+1$. 
the minimal number of superfluous variables is $3$ and we can take $f_j=x^{k-j}y^j$ for $j=0,\ldots,k$ and $f_{k+1}=z^k$.

Therefore $f = x^ku^e+x^{k-1}yu^{e-1}v+\ldots+y^ku^{e-k}v^k+z^ku^{e-k-1}v^{k+1}$ is a GNP-polynomial of type $(2,2,k,e)$. 
Hence, $\dim A_1 \geq 5$.

The maximal number of linearly independents $g_j \in \K[u,v]_e$ is $\dim_{\K} \K[u,v]_e = e+1$, so the maximal number of superfluous variables is $e+1$, and we can take 
$ f_j=x_j^k$ for $j=0,\ldots k$.
Therefore $\dim A_1 \leq e+3$ and all intermediate values are achieved.
\end{proof}
\medskip

\begin{thm}\label{thm:watanabecomponent} 
 For each $N \geq 4$, $d \geq 3$ and $1 \leq k < \frac{d}{2}$ there are GNP-polynomials $f=f_1g_1+\ldots+f_sg_s$ of 
 type $(m,n,k,e)$ with $N=m+n$ and $\deg(f)=d=e+k$. 
\end{thm}

\begin{proof} 
 Following  Definition \ref{defin:Am} we easily see that $\mathcal{A}_{k,e}^m=\{a_m,a_m+1,\ldots,b_m\}$. Since, by Lemma \ref{lema:m=2}, 
 $\mathcal{A}^2_{k,e}=\{5,6,\ldots,e+3\}$, it is enough to prove that $a_{m+1} \leq b_m$ for all $m\geq 2$.

 To verify the inequality we will compute $b_m$ and an upper bound $A_{m+1}$ for $a_{m+1}$, such that $a_{m+1}\leq A_{m+1} \leq b_m$. \\
 \begin{enumerate}
  \item Computation of $b_m$.\\
  Fixed $m,k,e$, to maximize $\dim A_1=m+n+1$ we must maximize $n$. Let $\{g_1,\ldots,g_s\}$ be complete basis of $\K[u_1,\ldots,u_m]_e$ (for example, the standard basis), then $s=\binom{m-1+e}{e}$.
  Let $f_j = x_j^k$ for $j=1,\ldots,\binom{m-1+e}{e}$. Since $s=\binom{m-1+e}{e}>\binom{m-1+k}{k}$, by Proposition \ref{prop:formacanonica}, $f$ is a GNP-form and $$b_m=m+\binom{m-1+e}{e}.$$
  \item An upper bound for $a_{m+1}$.\\ 
  
  We construct an example to obtain an upper bound for $a_{m+1}$. 
Let $\{f_1, \ldots, f_s\}$ be a basis of $\K[x_0,\ldots,x_{m+1}]_k$ and let $\{g_1, \ldots, g_s\} \subset \K[u_1,\ldots,u_{m+1}]_e$ be linearly independent set. This is possible because 
$s = \binom{m+1+k}{k} \leq  \binom{m+e}{e}$ for $e > k$. By Proposition \ref{prop:formacanonica}, $f$ is a GNP-polynomial of type $(m+2,m+1,k,e)$. Therefore, choosing the $g_i$ depending on all the variables $u_j$,
we get $\dim A_1=2m+3$, yielding
  $$a_{m+1}\leq 2m+3.$$
 \end{enumerate}
Notice that $a_{m+1} \leq 2m+3 \leq m+\binom{m-1+e}{e}=b_m$ for all $e\geq 2$ and for all $m \geq 2$. 
Indeed, $$\binom{m-1+e}{e}=\binom{m+(e-1)}{(e-1)+1} > \binom{m+e-1}{1} +1=m+e\geq m+2.$$  
Therefore $\binom{m-1+e}{e} \geq m+3 \Rightarrow b_m \geq 2m+3$. \\
The result now follows from the fact that $$\displaystyle \bigcup_{m \geq 2} \mathcal{A}_{k,e}^m = \{5,6,\ldots\}.$$
 \end{proof}

 \begin{rmk}\rm The GNP-polynomials are deeply connected with Gordan-Noether's ones and they generalize Perazzo's ones. For this reason they only appear for $N \geq 4$, and the case $k=1$ 
 is also covered. In general a GNP-polynomial of type $(m,n,k,e)$ with $k>1$ does not have $\hess^{k-1} =0$ as one can check directly in many examples. 
 Furthermore for $k>1$ one can prove that the general GNP-polynomial of type $(m,n,k,e)$ has $\hess_f \neq 0$. As a matter of fact we have been  unable to construct a
 GNP-polynomial of type $(m,n,k,e)$ such that $\hess_f^j=0$ for some $j <k$.
 \end{rmk}

Summarizing the results of this section we have proved the following generalization of Gordan Noether's theorem (c.f. Theorem \ref{thm:GN2}).

\begin{cor}\label{cor:GNgeneral}
 For each pair $(N,d)\not\in\{(3,3), (3,4)\}$ with  $N \geq 3$ and with $d \geq 3$, and for each $1 \leq k < \frac{d}{2}$ there exist 
 irreducible polynomials $f \in \K[x_0,\ldots,x_n,u_1,\ldots,u_m]$, where $N=m+n$, such that the hypersurface $X = V(f) \subset \P^N$ is not a cone and $\hess_f^k = 0$.  
\end{cor}

\section{Artinian Gorenstein algebras that do not satisfy the Lefschetz properties}

The goal of this section is to apply the previous results to construct Artinian Gorenstein algebras that do not satisfy the Lefschetz properties.
The link between these two subjects is Theorem \ref{WHessian} that will be restated now in a slightly different way. 

\begin{thm}\label{WHessiancor} {\rm (\cite{Wa1} and \cite{MW})} Let $A=Q/\operatorname{Ann}(f)$ be a standard graded Artinian Gorenstein algebra with 
$f\in\mathbb K[x_0,\ldots, x_N]_d$, suppose that $(\operatorname{Ann}(f))_1=0$. Then:
\begin{enumerate}
\item $A\simeq \frac{Q}{\operatorname{Ann}(f)}$
satisfies the $SLP$ if and only if $\hess^k_f\neq 0$ for every $k=1,\ldots,[d/2]$.
\medskip
\item If $d\leq 4$, then $A$ satisfies the $SLP$
if and only if $\hess(f)\neq 0$. In particular for $N\leq 3$, every such $A$
satisfies the $SLP$.
\end{enumerate}
\end{thm}

\begin{cor} \label{cor:slp} For each pair $(N,d)\not\in\{(3,3), (3,4)\}$ with  $N \geq 3$ and with $d \geq 3$, there exist standard graded 
Artinian Gorenstein algebras $A = \displaystyle \oplus_{i=0}^d A_i$ of codimension $\dim A_1=N+1 \geq 4$ and socle degree $d$ that do not satisfy the Strong 
Lefschetz Property. Furthermore, for each $L \in A_1$ we can choose arbitrarily the level $k$ where the map 
$$\bullet L^{d-2k} : A_k \to A_{d-k}$$ 
is not an isomorphism.
\end{cor}

\begin{proof} It is just a version of Corollary \ref{cor:GNgeneral} with a view of Theorem \ref{WHessiancor}.
\end{proof}
\medskip

\begin{cor} \label{cor:wlpodd} For each pair $(N,d)\neq (3,3)$ with  $N \geq 3$ and odd $d = 2q+1 \geq 3$, there exist standard graded Artinian Gorenstein algebras $A = \displaystyle  \oplus_{i=0}^d A_i$ with 
$\dim A_1=N+1$ and socle degree $d$ with unimodal Hilbert vector and that do not satisfy the Weak Lefschetz Property. 
\end{cor}

\begin{proof} 
Since $d=2q+1$ is odd, we can take $k=q$ in  Corollary \ref{cor:slp} so that $\bullet L : A_q \to A_{q+1}$ is not an isomorphism for all $L \in A_1$.
 Since $d-q=q+1$, $\dim A_q=\dim A_{q+1}$ and the map has not maximal rank. To conclude the proof we must show that we can choose $f$ in such a way that 
 the Hilbert vector $\operatorname{Hilb}(A)$ is unimodal.
 
  We shall consider two cases, according to $N$ is even or odd. 
 In both cases we use \\$\{M_i|i=1,\ldots,\nu = \binom{m-1+q}{q}\}$ to denote the standard basis of $\K[u_1,\ldots,u_m]_{q}$ lexicographically ordered. 
 So, we have $M_1=u_1^q$ and $M_{\nu}=u_m^q$. We take $s=\nu+1$, we set $\underline{x}=(x_1,\ldots,x_m)$ and $\underline{u}=(u_1,\ldots,u_m)$ in the next constructions we 
 denote $M_i(\underline{u})=M_i$ and $M_i(\underline{x})$ the equivalent monomial in the variables $x_1,\ldots,x_m$.
 \begin{enumerate}
  \item[(i)] {\bf First case: $N$ even.} Set $d=2q+1$, $n=m \geq 2$, $N=2m \geq 4$. Consider 
  $$f=x_0^qu_1^{q+1}  + \displaystyle \sum_{i=1}^{\nu} M_i(\underline{x})M_i(\underline{u})u_m.$$
  For each $k=1,\dots,\lfloor\frac{d}{2}\rfloor=q$, all differentials of the form $X_1^{a_1}\ldots X_m^{a_m}U_1^{b_1}\ldots U_m^{b_m}$ with $a_1+\ldots+a_m+b_1+\ldots+b_m=k$ 
  are linearly independent. The remaining differentials in $A_k$ are $X_0U_1^{k-1}, \ldots,X_0^{k-1}U_1,X_0^k$. Hence, 
  $$h_k=k+\binom{N-1+k}{N-1}, \ \text{for}\ k=1,\ldots,q.$$
  Therefore $\operatorname{Hilb}(A)$ is unimodal. 
  \item[(ii)] {\bf Second case: $N$ odd.} Set $d=2q+1$, $n=m+1\geq 3$, $N=2m +1 \geq 5$. Consider 
  $$f=x_0^qu_1^{q+1} + x_{m+1}^qu_m^{q+1} + \displaystyle \sum_{i=1}^{\nu - 1} M_i(\underline{x})M_i(\underline{u})u_m.$$
 For each $k=1,\dots,\lfloor\frac{d}{2}\rfloor=q$, all differentials of the form $X_1^{a_1}\ldots X_m^{a_m}U_1^{b_1}\ldots U_m^{b_m}$ with $a_1+\ldots+a_m+b_1+\ldots+b_m=k$ 
  are linearly independent. The remaining differentials in $A_k$ are $X_0U_1^{k-1}, \ldots,X_0^{k-1}U_1,X_0^k$ and $X_{m+1}U_m^{k-1}, \ldots,X_{m+1}^{k-1}U_m,X_{m+1}^k$. Hence, 
  $$h_k=2k+\binom{N-1+k}{N-1}, \ \text{for}\ k=1,\ldots,q.$$
  Therefore $\operatorname{Hilb}(A)$ is unimodal. 
\end{enumerate}

\end{proof}

\begin{rmk}\rm We want to recall that an Artinian Gorenstein algebra has the WLP if and only if the map $\bullet L:A_i \to A_{i+1}$ is injective in the middle (see \cite[Proposition 2.1]{MMN}).\\ 
For $d=2q+1$ odd we can take $i=q$ and consider only the map $\bullet L:A_q \to A_{q+1}$. In this case, by duality, the injectivity of $\bullet L$ gives us an isomorphism and it is equivalent 
to the non vanishing of  $\hess_f^{q}$.\\
On the other hand, for $d=2q$ even, we can take $i=q-1$ and consider the map $\bullet L:A_{q-1} \to A_q$. Notice that 
the non-vanishing of the Hessian $\hess_f^{q}$ gives us an isomorphism $\bullet L^2:A_{q-1}\to A_{q+1}$ which implies the injectivity of $\bullet L:A_{q-1} \to A_q$.
In the even case, the converse is not true. The injectivity of $\bullet L:A_{q-1} \to A_q$ implies the surjectivity of $\bullet L:A_{q} \to A_{q+1}$, but it does not imply an 
isomorphism $\bullet L^2:A_{q-1}\to A_{q+1}$. In this case we need more than the vanihsing of the $(q-1)$-th Hessian to get the failure of the WLP. 
The next example is of this type. 
\end{rmk}

\begin{ex}\rm
  Let $f=xu^3+yu^2v+zuv^2+v^4 \in R = \K[x,y,z,u,v]$ and let $A = Q/\Ann(f)$. Since $\hess_f=0$, the map $\bullet L^2:A_1 \to A_3$ is not an isomorphism. 
  On the other hand, the map $\bullet L:A_1 \to A_2$ is injective for $L=U+V$, as one can easily check.
\end{ex}

\begin{thm}\label{prop:44}
All standard graded Artinian Gorenstein $\K$-algebra $A = Q/\operatorname{Ann}(f)$ of codimension $5$ and socle degree $4$, with $\K$ an algebraically closed 
field of characteristic zero, satisfy the WLP.
\end{thm}

\begin{proof} By hypothesis $f \in \K[x_0,\ldots,x_4]$ and it depends on all the variables. 

If $\hess_f \neq 0$, then $$\bullet L^2: A_1 \to A_3$$ is an isomorphism, hence $\bullet L : A_1 \to A_2$ is injective and $\bullet L : A_2 \to A_3$ is surjective. 
 Therefore the result follows. 
 
  If $\hess_f = 0$, we claim that $f$ must be a reduced polynomial. On the contrary, if we take $\tilde{f} = \sqrt{f}$, the radical of $f$, 
then $\tilde{f}$ does not define a cone, $\deg(\tilde{f})=2,3$ and $\hess_{\tilde{f}}=0$ by Theorem \cite[Thm. 1]{DP}. If $\deg(\tilde{f})=2$ we have a contradiction, since Hesse's claim is true for quadratic polynomials. 
The other possibility is that $\deg(\tilde{f})=3$. In this case, $\tilde{f}=l_1l_2l_3$ is a product of three independent linear forms, which defines a cone, and we have a contradiction again.
So we can assume that $f$ is a reduced polynomial. By the classification Theorem of Gordan and Noether in $\P^4$, Theorem \ref{thm:GN3}, up to a projective transformation $f$ 
must be of the form
 $$f=x_0f_0+x_1f_1+x_2f_2+h,$$
 where $f_i \in \K[u,v]_3$ and $h \in \K[u,v]_4$. It is easy to see that if we change $h \in \K[u,v]_4$ the Hessian is still zero, so we can suppose that $f$ is irreducible. 
 Consider the map $\phi:\P^1 \dashrightarrow \P^2$ given by $\phi(u:v)=(f_0:f_1:f_2)$.
 The image of $\phi$, $Z = \overline{\phi(\P^1)}$ is a rational curve of degree two or three. In fact, it is a projection of the twisted cubic $\mathcal{V}_3(\P^1) \subset \P^3$
 from a point. We have only three possibilities:
 \begin{enumerate}
  \item[(i)] Projection from an internal point. In this case $Z \subset \P^2$ is a conic. Up to projective transformations $Z= V(z^2-xy) \subset \P^2$, 
  $f_0=u^3, f_1=uv^2, f_2 = u^2v$. 
  In this case $$f=x_0u^3+x_1uv^2+x_2u^2v+h(u,v).$$
  Taking $L=U+V \in A_1$ one can verify directly that $\bullet L: A_1 \to A_2$ is injective. The map $\bullet L : A_2 \to A_3$ is surjective since the image of $\{X_1U,X_1V,U^2,UV,V^2\}$ 
  generates $A_3$. Therefore $A$ satisfy the WLP.
 
 \item[(ii)] An external projection whose center belongs to the tangent surface of the twisted cubic, $T \mathcal{V}_3(\P^1)$. In this case $Z \subset \P^2$ is a cuspidal cubic. 
 Up to a projective transformation $Z = V(zy^2-x^3) \subset \P^2$ and $f_0=u^2v,f_1=u^3,f_2=v^3$. In this case 
 $$f=x_0u^2v+x_1u^3+x_2v^3+h(u,v).$$
 Taking $L=U+V \in A_1$ one can check that $\bullet L: A_1 \to A_2$ is injective. The map $\bullet L : A_2 \to A_3$ is surjective since the image of $\{X_0U,X_0V,U^2,UV,V^2\}$ 
  generates $A_3$. Therefore $A$ satisfy the WLP.
 \item[(iii)] A general external projection. In this case $Z \subset \P^2$ is a nodal cubic curve. Up to a projective transformation $Z = V(zy^2-x^2(x+z)) \subset \P^2$ and 
 $f_0=v(u^2-v^2), f_1=u(u^2-v^2), f_2=v^3$. In this case  
  $$f=x_0v(u^2-v^2)+x_1u(u^2-v^2)+x_2v^3+h(u,v).$$
 Taking $L=V \in A_1$ one can check that $\bullet L: A_1 \to A_2$ is injective and $\bullet L : A_2 \to A_3$ is surjective since the image of $\{X_0U,X_0V,U^2,UV,V^2\}$ 
  generates $A_3$. Therefore $A$ satisfy the WLP.
 \end{enumerate}

 \end{proof}

\begin{lema}\label{lema:wlp}
 Let $f \in R = \K[x_0,\ldots,x_n,u_1,\ldots,u_m]$ of degree $\deg(f) = k+d$ with $k<d$. Suppose there exist $\alpha_1,\ldots,\alpha_s \in A_k$ linearly independent differential operators such that for all 
 $L \in A_1$ we have  $f_{L \alpha_i} \in \K[u_1,\ldots,u_m]$. If $s >\binom{m+d-2}{d-1}$, then the map $\bullet L: A_k \to A_{k+1}$ is not injective. 
\end{lema}

\begin{proof} Let $Q=\K[X_0,\ldots,X_n,U_1,\ldots,U_m]$ be the ring of differentials and consider the multiplication $\bullet L: Q_k \to Q_{k+1}$. 
Let $A_k = Q_k/Q_k \cap \operatorname{Ann}(f)$. Consider also the evaluation map $ev: Q_k \to R_{D-k}$, sending $\alpha$ to $f_{\alpha}$. 
 An element $\alpha \in A_{k}\setminus \{0\}$ is in the kernel of the multiplication map $\bullet L: A_{k} \to A_{k+1}$ if and only if there is a representative, 
 which by abusing  notation we denote by $\alpha \in Q_k$, whose image under the composition $\phi_k:Q_k \to Q_{k+1} \to R_{d-1}$ is zero. \\
 Choose for $j=1,\ldots,s$, $\alpha_j \in Q_k$ a representative whose image in $A_k$ is $\alpha_j$. 
 Let $$W = \bigoplus_{i=1}^s \K \alpha_i \subset Q_k.$$ By  hypothesis $\dim W = s>\binom{m+d-2}{d-1}$. Since the image of the restriction of $\phi_k$ to $W$ 
 lies in $\K[u_1,\ldots,u_m]_{d-1}$ and since $ \dim \K[u_1,\ldots,u_m]_{d-1} = \binom{m+d-2}{d-1}$ the result follows.
\end{proof}

\begin{lema}\label{lema:unimodalvariaveisseparadas}
Let $f \in R=\K[x_0,\ldots,x_n,u_1,\ldots,u_m]_d$ and suppose $f=g+h$ with $g \in \K[x_0,\ldots,x_n]_d$ and $h \in \K[u_1,\ldots,u_m]_d$. 
 Let $Q$ be the associated ring of differential operators and set $A(f)=Q/\Ann(f)$, $A(g)=Q/\Ann(g)$ and $A(h)=Q/\Ann(h)$. 
 Then, the Hilbert vector of $A(f)$ is given by $$\dim(A_k(f))=\dim(A_k(g))+\dim(A_k(h)).$$
 For $k=1,\ldots,m+n$. In particular, if $\Hilb(A(g))$ and $\Hilb(A(h))$ are unimodal, then $\Hilb(A(f))$ is also unimodal.
\end{lema}
\begin{proof}
Notice that $\Ann(f) = \Ann(g)\cap \Ann(h)$. \
 For $I,J \subset  Q$ homogeneous ideals, we have for each degree $k=1,\ldots,d-1$ 
 $$(\frac{Q}{I\cap J})_k \simeq (\frac{Q}{I})_k \oplus (\frac{Q}{J})_k$$ as $\K$-vector spaces.  
\end{proof}

We now are in position to prove one of our main results.

\begin{thm} \label{thm:wlp}
 For each pair $(N,d)\not\in\{(3,3),(3,4),(4,4), (3,6)\}$ with $N \geq 3$ and with $d \geq 3$  there exist standard graded Artinian Gorenstein 
 algebras $A = \displaystyle  \oplus_{i=0}^d A_i$ of codimension $N+1$ and socle degree $d$, with unimodal Hilbert vector, $\Hilb(A)$ and that do not satisfy the Weak Lefschetz Property.
\end{thm}
\begin{proof} For $N=3$ and $d=3,4$ the impossibility comes from Corollary \ref{WHessiancor}. 
The case $(3,6)$ follows by \cite[Cor. 3.12]{BMMNZ} which includes also the cases $(3,3)$ and $(3,4)$. In the case $(4,4)$ Proposition \ref{prop:44} yields that $A$ satisfies the WLP. 

Corollary \ref{cor:wlpodd} yields the result for odd $d$ so we can restrict ourselves to the even case $d=2q \geq 4$ and $N \geq 3$.

For $d=4$ we assume $N \geq 5$. Consider the polynomials of the form 
$$f=x_2u^3+x_3u^2v+x_4uv^2+x_5v^3+g(u,v)+h(x_6,\ldots,x_N)$$
and notice that $X_2,X_3,X_4,X_5 \in A_1$ are linearly independent. Indeed, $f_{X_2}=u^3$, $f_{X_3}=u^2v$, $f_{X_4}=uv^2$ and $f_{X_5}=v^3$. Notice also that 
$f_{LX_i} \in \K[u,v]_2$. In fact, if $$L=a_0U+a_1V+a_2X_2+a_3X_3+a_4X_4+a_5X_5+a_6X_6+\ldots+a_NX_N \in\ A_1,$$ then 
$LX_i = a_0UX_i+a_1VX_i \in A_2$ for $i=2,3,4,5$ and the claim follows. 
Since $s=4>3=\dim \K[U,V]_2$ we can apply Lemma \ref{lema:wlp} to deduce that  the map $\bullet L : A_1 \to A_2$ is not injective, proving 
that $A$ does not satisfy the WLP. On the other hand, taking $\tilde{f} = x_2u^3+x_3u^2v+x_4uv^2+x_5v^3+g(u,v)$, one can check that $\Hilb(A(f_1))=(1,6,6,6,1)$ for all 
$g \in \K[u,v]$. For a general $h \in \K[x_6,\ldots,x_N]$, $\Hilb(A(h))$ is unimodal. Hence, by Lemma \ref{lema:unimodalvariaveisseparadas}, $\Hilb(A(f))$ is unimodal. 

For $d=6$ we assume $N \geq 4$ and take the exceptional form given by 
$$f=x_2u^2v^3+x_3u^4v+x_4uv^4+g(u,v)+h(x_2,x_3,\ldots,x_N)$$
We have five linearly independent second order differentials $\alpha_1=X_2U, \alpha_2=X_2V,\alpha_3=X_3U,\alpha_4=X_3V, \alpha_5= X_4U \in A_2$ such that
$f_{\alpha_1}=uv^3$, $f_{\alpha_2}=u^2v^2$, $f_{\alpha_3}=u^4$, $f_{\alpha_4}=u^3v$ and $f_{\alpha_5}=v^4$. 
For all $$L = aU+bV+a_2X_2+a_3X_3+\ldots+a_NX_N$$ we have $L\alpha_i=aU\alpha_1+bV\alpha_i$ for $i=1,\ldots,5$. Therefore, by Lemma \ref{lema:wlp},
the map $\bullet L: A_2 \to A_3$ is not injective so that $A$ does not satisfy the WLP. We claim that $\Hilb(A(f))$ is unimodal. In fact, taking 
$f_1=x_2u^2v^3+x_3u^4v+x_4uv^4+g(u,v)$, $\Hilb(A(f))= (1,5,8 , 8, 8,5,1)$ for all $g \in \K[u,v]$. Choosing $h$ in such a way $\Hilb(A(h))$ is unimodal, 
we conclude that $\Hilb(A(f))$ is unimodal by Lemma \ref{lema:unimodalvariaveisseparadas}.  

For $d=2q \geq 8$, we assume $N \geq 3$. We investigate the following maps:
$$A_{q-1} \to A_q \to A_{q+1}$$
If $\dim A_q < \dim A_{q-1}$, then the Hilbert vector of $A$ is not unimodal and hence the algebra does not satisfy the WLP. So we can suppose  $\dim A_q \geq \dim A_{q-1}$ and 
we prove that $\bullet L: A_{q-1} \to A_q$ is not injective, which implies that  $A$ does not satisfy the WLP.
Let $f$ be the exceptional form
$$f=x_2u^{q-2}v^{q+1}+x_3u^{2q-3}v^2+g(u,v)+h(x_4,x_5,\ldots,x_N)$$
Letting  $\alpha_i=X_2U^iV^{q-2-i} \in A_{q-1}$ for $i=0,\ldots,q-2$,  we have $f_{\alpha_i}=a_iu^{q-2-i}v^{3+i}$.  Let 
$\beta_j=X_3V^jU^{q-2-j} \in A_{q-1}$ for $j=0,1,2$ so that, after remarking that  $2 \leq q-2$, we deduce $f_{\beta_j}=b_ju^{q-1+j}v^{2-j}$. 
Since the monomials $f_{\alpha_i}$ and $f_{\beta_j}$ are distinct for $i=0,\ldots,q-2$ and $j=0,1,2$,  they are linearly independent in $A_{q-1}$ so that
$W=\oplus \K \alpha_i \oplus \K \beta_i \subset Q_{q-1}$ has dimension $q+2$. For all $L \in A_1$ it is easy to check that $f_{L \alpha_i},f_{L \beta_i} \in \K[u,v]_{q}$. 
Since $\dim \K[u,v]_{q}= q+1$, we can apply 
Lemma \ref{lema:wlp} to deduce that the map $\bullet L:A_{q-1} \to A_q$ is not injective, proving the result. By other side $f_1=x_2u^{q-2}v^{q+1}+x_3u^{2q-3}v^2+g(u,v)$ 
define a Artinian Gorenstein algebra of codimension $4$ and such that $\dim A_2 = 7$, hence, by the main result of \cite{IS}, $\Hilb(A(f_1))$ is unimodal. 
Again, by Lemma \ref{lema:unimodalvariaveisseparadas} we conclude that for a general choice of $h$, $\Hilb(A(h))$ is unimodal, thus $\Hilb(A(f))$ is unimodal.
\end{proof}
 \medskip

 \section*{Appendix: Forms with vanishing Hessian, an overview} \label{ap}
 
 The aim of this appendix is to recall classical results and constructions of hypersurfaces with vanishing hessian whose original work was writen in German (see \cite{GN}) and Italian (see \cite{Pe, Pt1, Pt2, Pt3}). 
 Some of this classical work was revisited in \cite{CRS, GR, Wa1, Lo, GRu, Ru, dB}.  
 
 The interest in forms with vanishing hessian starts with Hesse's claim (\cite{Hesse1,Hesse2}) stating that a form of degree $d$, $f \in \K[x_0,x_1,\ldots,x_N]$ has vanishing hessian, $\hess_f=0$, if and only if 
 there exists a linear change of coordinates $\pi$ such that the transformed form $f_{\pi}$ does not depend on all the variables. Since the vanishing of the hessian is invariant by a linear change of coordinates, the ``if'' implication is trivial. 
 In degree $d=2$ Hesse's claim is trivial by diagonalizing the quadratic form. From now on we shall assume $d \geq 3$. Let $f$ be a reduced polynomial and let $X = V(f) \subset \P^N$ the hypersurface. 
 From a geometric point of view Hesse's claim can be restated as $\hess_f = 0 $ if and only if $X$ is a cone. \\ 
Hesse's claim is not true in general as it was observed by Gordan and Noether (\cite{GN}). 
More precisely,  in \cite{GN} the authors proved that Hesse's claim is true for $N\leq 3$ and they produced a series of counterexample for any $N \geq 4$ and for any $d \geq 3$. 
The easiest counterexample to Hesse's claim is $f=xu^2+yuv+zv^2 \in \K[x,y,z,u,v]$ and it was explicitly posed by Perazzo in \cite{Pe}, who called it {\it un esempio semplicissimo}. 
A modern proof of the next result can be found in (\cite{dB,Lo,GR,Wa2,Wa3Ru}).

\begin{thm}\label{thm:GN} \cite{GN} Let $X = V(f) \subset \P^N$, $N \leq 3$, be a hypersurface such that $\hess_f=0$. Then $X$ is a cone.
\end{thm}

In \cite{GN} the authors produced a series of counterexamples to Hesse's claim for each $N \geq 4$ and for each degree $d \geq 3$. 
The key point of the construction was to figure out that the vanishing of the Hessian is equivalent to the algebraic dependence among the partial derivatives 
(see {\it loc. cit.}). On the other hand, to be a cone is equivalent to the linear dependence among the partial derivatives. This result is sometimes referred as Gordan-Noether's criterion since it was implicitly used in \cite{GN}. 
A proof of it can be found in \cite{CRS} and in \cite[Chapter 7]{Ru}. 

\begin{prop} \label{prop:GNcriteria} \cite{GN} Let $f\in \K[x_0,\ldots,x_N]$ be a reduced polynomial and consider $X  = V(f) \subset \P^N$. Then

\begin{enumerate}
 \item[(i)] $X$ is a cone $ \Leftrightarrow f_{X_0},\ldots,f_{X_N} $ are linearly dependent;
 \item[(ii)] $\hess_f=0 \Leftrightarrow f_{X_0},\ldots,f_{X_N}$ are algebraically dependent.
\end{enumerate}

\end{prop}

\begin{thm}\label{thm:GN2} \cite{GN} For each $N \geq 4$ and $d \geq 3$ there exist irreducible hypersurfaces $X = V(f) \subset \P^N$, 
of degree $\deg(f) = d$, not cones, such that $\hess_f=0$.  
\end{thm}

\begin{proof}
 See Theorem \ref{thm:perazzo} and Theorem \ref{thm:permutti} for a short proof. 
\end{proof}



For the reader's convenience we recall the classical constructions of Gordan and Noether, (\cite{GN}), Permutti, (\cite{Pt1,Pt2,Pt3}) and Perazzo (\cite{Pe}) from an algebraic point of view. 

\begin{defin} Let $X = V(f) \in \P^N$, $N \geq 4$ be an irreducible hypersurface not a cone. We say that $X$ is a Perazzo hypersurface of degree $d$ if 
$N=n+m$, $n,m \geq 2$ and $f \in \K[x_0,\ldots,x_n,u_1,\ldots,u_m]$ is a reduced polynomial of the form 
$$f=x_0g_0+\ldots+x_ng_n+h$$
where $g_i \in \K[u_1,\ldots,u_m]_{d-1}$ for $i=0,\ldots,n$ are algebraically dependent but linearly independent and 
$h \in \K[u_1,\ldots,u_m]_{d}$. The polynomial $f$ is called Perazzo polynomial.  
\end{defin}

\begin{thm}\cite{GN,Pe} \label{thm:perazzo} Perazzo hypersurfaces are not cones and have vanishing Hessian. 
\end{thm}

\begin{proof} Since $f_{X_i} = g_i$ for $i=0,\ldots,n$ are algebraically dependent by hypothesis, by Proposition \ref{prop:GNcriteria}, we have $\hess_f=0$. 
\end{proof}

\begin{rmk}\rm 
Notice that if $n+1 > m$, then $g_i$ for $i=0,\ldots,n$ are algebraically dependent automatically.  
Perazzo original hypersurfaces are of degree $3$ and he constructed a series of cubic hypersurfaces in $\P^N$ for arbitrary  $N \geq 4$ with vanishing Hessian and not cones. 
These hypersurfaces are, modulo projective transformations, all cubic hypersurfaces with vanishing Hessian and not cones in $\P^N$ for $N=4,5,6$, (see \cite{Pe,GRu}).
\end{rmk}

\begin{defin}
 Let $R = \K[x_0,\ldots,x_n,u_1,\ldots,u_m]$. Let $Q= \displaystyle \sum_{i=0}^nx_0g_i \in R$ be a form of degree $e$ with 
 $g_i \in \K[u_1,\ldots,u_m]_{e-1}$ for $i=0,\ldots,n$ algebraically dependent but linearly independent.
 Let $\mu = \lfloor \frac{d}{e} \rfloor$. Let $P_j \in \K[u_1,\ldots,u_m]_{d-je}$ for $j =0,\ldots,\mu$. We say that 
 $$f=\displaystyle \sum_{j=0}^{\mu}Q^jP_j$$
 is a Permutti polynomial of type $(m,n,e)$. A hypersurface $X=V(f) \subset \P^N$, not a cone is called a Permutti hypersurface if $f$ is a reduced Permutti polynomial. 
\end{defin}

\begin{thm}\cite{Pt1, Pt2} \label{thm:permutti} Permutti hypersurfaces are not cones and have vanishing Hessian. 
\end{thm}

\begin{proof} We have $f_{X_i} = (\displaystyle \sum_{j=1}^{\mu}jQ^{j-1}P_j)g_i=Gg_i$ for $i=0,\ldots,n$. Since, for $i=0,\ldots,n$, $g_i$ are algebraically dependent, 
$f_{X_i}$ are too. Therefore, by Proposition \ref{prop:GNcriteria}, we have $\hess_f=0$. 
\end{proof}

Finally we present the original Gordan and Noether hypersurfaces with a slight simplification. A modern proof that GN=polynomials have vanihsing hessian can be found in  \cite{CRS,Ru}.

\begin{defin} \label{def:GN} Let $R = \K[x_0,\ldots,x_n,u_1,\ldots,u_m]$. For $l=1,\ldots,s$ and for $i=0,\ldots,n$ let $\Phi_{il} \in \K[y_0,\ldots,y_r]$ and $\Psi_{k}^{il} \in \K[u_1,\ldots,u_m]$ 
be homogeneous polynomials. 
Let $g_{il} \in \K[u_1,\ldots,u_m]_{e-1}$ be given by $g_{il} = \Phi_{il}(\Psi_0^{il},\ldots,\Psi_r^{il})$, with $0 \leq i \leq n$ and $1 \leq l \leq s$ where $s=n-r$. 
Let $Q_l = x_0g_{0l}+\ldots+x_ng_{nl}$ with $l=1,\ldots,s$. 
Let $d>e$ and $\mu = \lfloor \frac{d}{e}\rfloor$. Let $P_j(z_1,\ldots,z_s,u_1,\ldots,u_m)$ for $j =0,\ldots,\mu$ be bi-forms of bi-degree $(j,d-ej)$. 
A GN hypersurface of type $(m,n,r,e)$ is defined by a polynomial 
 $$f=\displaystyle \sum_{j=0}^{\mu}P_j(Q_1,\ldots,Q_s,u_1,\ldots,u_m).$$
\end{defin}



\begin{rmk}\rm It is easy to see that a Perazzo hypersurface is a Permutti hypersurface with $\mu=1$. Notice also that a GN hypersurface of type $(m,n,n-1,e)$ must have  
$s=1$, hence it is a Permutti hypersurface of type $(m,n,e)$. We have presented the constructions in an increasing order of generality and complexity 
but the chronological order is actually \cite{GN}, \cite{Pe} and \cite{Pt1}. 
\end{rmk}

The main result of Gordan-Noether in \cite{GN} is the following one. A geometric proof in modern terms can be found in \cite{GR, Ru}.

\begin{thm}\label{thm:GN3} \cite{GN}
 Let $X =V(f) \subset \P^4$ be a reduced hypersurface, not a cone, having vanishing Hessian. The                                                                                                                                                                                                                                                                                                                                                                                                                                                                                                                                                                                                                                                                                                                                                                                                                                                                                                                                                                                                                                                                                                                                                                                                                                                                                                                                                                                  n $f$ is a GN polynomial of type $(2,2,1,e)$ or equivalently 
 a Permutti polynomial of type $(2,2,e)$. 
\end{thm}

{\bf Acknowledgments}.
I wish to thank Francesco Russo for his suggestions, corrections and conversations on this subject among many others.

I wish to thank the referee for the helpful corrections and suggestions on the detailed revision of a previous version of this work.  

I am also grateful to Giuseppe Zappal\`a for useful conversations and to the participants to the weekly Commutative Algebra/Algebraic Geometry Seminar in Catania, where I was 
introduced to this theme of research in the framework of the Research Project of the University of Catania FIR 2014 "Aspetti geometrici e algebrici della Weak e Strong Lefschetz Property". 

I wish to thank also Junzo Watanabe for his suggestions leading to a significative improvement of this note. \\

Partially supported  by the CAPES postdoctoral fellowship, Proc. BEX 2036/14-2.

\end{document}